\documentclass[a4paper,10pt]{amsart}
\usepackage{amsmath,amsthm,amssymb,amsfonts,enumerate,color}

\newcommand{\norm}[1]{\left\Vert#1\right\Vert}
\newcommand{\abs}[1]{\left\vert#1\right\vert}
\newcommand{\set}[1]{\left\{#1\right\}}
\newcommand{\Real}{\mathbb{R}}

\renewcommand{\H}{\mathcal{H}}
\renewcommand{\L}{\mathcal{L}}

\newcommand{\N}{\mathbb{N}}

\newcommand{\R}{\mathbb{R}}

\newcommand{\Dom}{\operatornamewithlimits{Dom}}
\newcommand{\dive}{\operatornamewithlimits{div}}

\newtheorem{thm}{Theorem}[section]

\newtheorem{lem}[thm]{Lemma}

\theoremstyle{definition}
\newtheorem{defn}[thm]{Definition}
\newtheorem{rem}[thm]{Remark}

\numberwithin{equation}{section}

\author[P. R. Stinga]{Pablo Ra\'ul Stinga}
\address{Departamento de Matem\'aticas y Computaci\'on\\
         Universidad de La Rioja\\
         26004 Logro\~no, Spain}
\email{pablo-raul.stinga@unirioja.es}

\author[C. Zhang]{Chao Zhang}
\address{School of Mathematics and Statistics \\
          Wuhan University \\
          430072 Wuhan, China \\ --AND-- Departamento de Matem\'aticas \\
          Facultad de Ciencias \\
          Universidad Aut\'onoma de Madrid \\
          28049 Madrid, Spain}
\email{zaoyangzhangchao@163.com}

\thanks{Research partially supported by Ministerio de Ciencia e Innovaci\'{o}n de Espa\~{n}a MTM2008-06621-C02-01. The first author was partially supported by grant COLABORA 2010/01 from Planes Riojanos de I+D+I. The second author was partially supported by National Natural Science Foundation of China No.11071190}



\begin{document}

%%%%%%%%%%%%%%%%%%%%%%%%%%%%%%%%%%%%%%%%%%%%%%%%%%%%%%
\title[Harnack's inequality]{Harnack's inequality for fractional nonlocal equations}
%%%%%%%%%%%%%%%%%%%%%%%%%%%%%%%%%%%%%%%%%%%%%%%%%%%%%%

%%%%%%%%%%%%%%%%%%%%%%%%%%%%%%%%%%%%%%%%%%%%%%%%%%%%%%
\begin{abstract}
We prove interior Harnack's inequalities for solutions of fractional nonlocal equations. Our examples include fractional powers of divergence form elliptic operators with potentials, operators
arising in classical orthogonal expansions and the radial Laplacian.
To get the results we use an analytic method based on a
generalization of the Caffarelli--Silvestre extension problem, the
Harnack's inequality for degenerate Schr\"odinger operators proved
by C. E. Guti\'errez, and a transference method. In this manner we
apply local PDE techniques to nonlocal operators. On
the way a maximum principle and a Liouville theorem for some fractional nonlocal equations are obtained.
\end{abstract}
%%%%%%%%%%%%%%%%%%%%%%%%%%%%%%%%%%%%%%%%%%%%%%%%%%%%%%

\maketitle

%%%%%%%%%%%%%%%%%%%%%%%%%%%%%%%%%%%%%%%%%%%%%%%%%%%%%%
\section{Introduction}
%%%%%%%%%%%%%%%%%%%%%%%%%%%%%%%%%%%%%%%%%%%%%%%%%%%%%%

Very recently, a great deal of attention was given to nonlinear
problems involving fractional integro-differential operators. These
problems arise in Physics (fluid dynamics, strange kinetics, anomalous transport) and Mathematical
Finance (modeling with L\'evy processes), among many other fields, see for instance
\cite{Caffa survey, Caffarelli-Salsa-Silvestre, Caffarelli-Vasseur, Jin-Li-Xiong, Shlesinger-Zaslavsky-Klafter, Silvestre} and the references therein. The main question is the regularity of solutions. One of the tools
that plays a crucial role in the regularity theory of PDEs is Harnack's inequality,
see for example \cite{Caffarelli-Salsa-Silvestre, Caffarelli-Silvestre, Davies, Fabes-Kenig-Serapioni, Gilbarg-Trudinger, Guti, Roncal-Stinga, Stinga-Torrea-CPDE, Tan-Xiong, Trudinger}.

In this paper we show interior Harnack's inequalities for solutions of nonlocal equations given by fractional powers of second order partial differential operators. The operators we consider are:
\begin{itemize}
    \item Divergence form elliptic operators $\L=-\dive(a(x)\nabla)+V(x)$ with bounded measurable coefficients $a(x)$ and locally bounded nonnegative potentials $V(x)$ defined on bounded domains;
    \item Ornstein-Uhlenbeck operator $\mathbf{O}_{\mathbf{}B}$ and harmonic oscillator $\H_{\mathbf{B}}$ on $\Real^n$;
    \item Laguerre operators $\mathbf{L}_\alpha$, $\mathbf{L}_\alpha^\varphi$, $\mathbf{L}_\alpha^\ell$, $\mathbf{L}_\alpha^\psi$ and $\mathbf{L}_\alpha^\L$ on $(0,\infty)^n$ with $\alpha\in(-1,\infty)^n$;
    \item Ultraspherical operators $L_\lambda$ and $l_\lambda$ on $(0,\pi)$ with $\lambda>0$;
    \item Laplacian on domains $\Omega\subseteq\Real^n$;
    \item Bessel operators $\Delta_\lambda$ and $S_\lambda$ on $(0,\infty)$ with $\lambda>0$.
\end{itemize}
For the full description of the operators see Sections \ref{Section:Reflection}, \ref{Section:Orthogonal} and \ref{Section:Laplacian-Bessel}. In general, all these operators $L$ are nonnegative, self-adjoint and have a dense domain $\Dom(L)\subset L^2(\Omega,d\eta)$, where $\Omega\subseteq\Real^n$, $n\geq1$, is an open set and $d\eta$ is some positive measure on $\Omega$. In Section \ref{Section:Extension} we show how the fractional powers $L^\sigma$, $0<\sigma<1$, can be defined by using the spectral theorem.

\

\noindent{\bf Theorem A} (Harnack's inequality for fractional equations)\textbf{.} \textit{Let $L$ be any of the operators listed above and $0<\sigma<1$. Let $\mathcal{O}$ be an open and connected subset of $\Omega$ and fix a compact subset $K\subset\mathcal{O}$. There exists a positive constant $C$, depending only on $\sigma$, $n$, $K$ and the coefficients of $L$ such that
$$\sup_Kf\le C \inf_Kf,$$
for all functions $f\in\Dom(L)$, $f\geq0$ in $\Omega$, such that
$L^\sigma f=0$ in $L^2(\mathcal{O},d\eta)$. Moreover, $f$ is a continuous function in $\mathcal{O}$.}

\

Theorem A is new, except for three cases: the Laplacian on $\Real^n$ (\cite[Theorem~5.1]{Caffarelli-Silvestre} and \cite[p.~266]{Landkof}), the Laplacian on the one-dimensional torus \cite[Theorem~6.1]{Roncal-Stinga} and the harmonic oscillator \cite[Theorem~1.2]{Stinga-Torrea-CPDE}. Harnack's inequality is well-known for divergence form Schr\"odinger operators with locally bounded potentials \cite{Guti}, see also \cite{Davies, Gilbarg-Trudinger, Trudinger}. For the non-divergence form operators listed above the result can be obtained by using our transference method of Section \ref{Section:Transference}. Very recently a Harnack's inequality for the fractional Laplacian with lower order terms was proved in \cite{Tan-Xiong}.

A novel proof of Harnack's inequality for the fractional Laplacian was given by L. Caffarelli and L. Silvestre by using the extension problem in \cite{Caffarelli-Silvestre}. Let us briefly explain it here. Consider $f:\Real^n\to\Real$ as in the hypotheses of Theorem A. Let $u(x,y)$ be the extension of
$f$ to the upper half space $\Real^{n+1}_+$ obtained by solving
$$
\begin{cases}
\dive(y^{1-2\sigma}\nabla u)=0, &\hbox{in} ~\Real^n\times(0,\infty); \\
u(x,0)=f(x), &\hbox{on}~\Real^n.
\end{cases}
$$
Let $\tilde{u}(x,y)=u(x,|y|)$, $y\in\Real$, be the reflection of $u$ to $\Real^{n+1}$. The
hypothesis $(-\Delta)^\sigma f=0$ in $\mathcal{O}$ implies that
$y^{1-2\sigma}u_y(x,y)\to0$ as $y\to0^+$, for all $x\in\mathcal{O}$. This is
used to show that $\tilde{u}$ is a weak solution of the degenerate
elliptic equation with $A_2$ weight
$$\dive(|y|^{1-2\sigma}\nabla\tilde{u})=0,\quad\hbox{in}~\mathcal{O}\times(-R,R)\subset\Real^{n+1},$$
for some $R>0$. Recall that a nonnegative function $\omega$ on $\Real^n$ is an $A_2$ weight if
$$\sup_{B\,\hbox{\tiny ball}}\left(\frac{1}{|B|}\int_B\omega\right)\left(\frac{1}{|B|}\int_B\omega^{-1}\right)<\infty.$$
Then the theory of degenerate elliptic equations by E. Fabes, C. Kenig and
R. Serapioni in \cite{Fabes-Kenig-Serapioni} says that $\tilde{u}$
satisfies an interior Harnack's inequality and it is locally H\"older continuous,
thus $f(x)=\tilde{u}(x,0)$ has the same properties.

The idea of \cite{Caffarelli-Silvestre} was also exploited in \cite{Stinga-Torrea-CPDE} for the case of the fractional harmonic oscillator $(-\Delta+|x|^2)^\sigma$ on $\Real^n$, under the additional assumption $f\in C^2$. In \cite{Stinga-Torrea-CPDE} a generalization of the extension problem was proved that applies to a general class of differential operators, and it was used to get the result for the harmonic oscillator. We observe that, instead of the theory of \cite{Fabes-Kenig-Serapioni}, Harnack's inequality for degenerate Schr\"odinger operators of C. E. Guti\'errez \cite{Guti} had to be applied.

To get Harnack's inequalities for fractional powers of the operators listed above we push further the
Caffarelli--Silvestre ideas. We proceed in two steps. First we use two tools: the extension
problem of \cite{Stinga-Torrea-CPDE} and Harnack's inequality for degenerate Schr\"odinger operators of C. E. Guti\'errez \cite{Guti}. These are enough to get Theorem \ref{Thm:Harnack}, from which the result for divergence form elliptic operators with potentials and some Schr\"odinger operators from orthogonal expansions is deduced. Secondly, we apply systematically a transference
method that permits us to derive the results for other operators
involving terms of order one and in non-divergence form. The
transference method is inspired in ideas from Harmonic Analysis of
orthogonal expansions, where it is used to transfer $L^p$
boundedness of operators, see for example \cite{AbuMST, AbuTorrea,
Gutierrez-Incognito-Torrea}. In that case, the dimension, the underlying measure and the parameters that define the operators play a significant role. Here we can obtain our estimates without any restrictions on dimensions or parameters.

Let us remark that in Theorem A we require the condition $f\geq0$ all over $\Omega$, which is needed to ensure that the solution to the extension problem $u$ is nonnegative in $\Omega\times(0,\infty)$. In fact, $u$ can be given in terms of the solution $e^{-tL}f$ of the $L$-heat diffusion equation, see Theorem \ref{Thm:Extension general} below, so we only would need the condition $e^{-tL}f\geq0$ in $\mathcal{O}$. Certainly it is sufficient to assume that $e^{-tL}$ is positivity-preserving (see \eqref{positivity} below), but this hypothesis is not strictly necessary.

As a by-product of our method, we obtain a Liouville theorem for fractional powers of divergence form elliptic operators on $\Real^n$, see Remark \ref{Rem:Liouville}. We also get a maximum and comparison principle for general fractional operators, see Remark \ref{Rem:comparison}.

In Section \ref{Section:Extension} we present the definition of fractional powers of differential operators, we get maximum and comparison principles and we state the extension problem of \cite{Stinga-Torrea-CPDE}. The method of reflections for proving Harnack's inequality for divergence form elliptic Schr\"odinger operators is given in Section \ref{Section:Reflection}. The transference method is explained in Section \ref{Section:Transference}.

The rest of the paper is concerned with the proof of Theorem A in each case. As the reader may notice, we have two sets of applications of our method: operators with discrete spectrum and operators with continuous spectrum. In the first set we have divergence form elliptic operators in bounded domains and classical operators related to orthogonal expansions in possibly unbounded domains (Sections \ref{Section:Reflection} and \ref{Section:Orthogonal}). In the second set (Section \ref{Section:Laplacian-Bessel}) we have the Laplacian (Fourier transform) and the Bessel operator (Hankel transform), that generalizes the radial Laplacian.

We will present most of the results about Harnack's inequalities in the case when the sets $K$ and $\mathcal{O}$ in Theorem A are balls
inside $\Omega$. In that situation the constant $C$ does not depend on the radius of the balls. By the standard covering argument \cite[Theorem~2.5]{Gilbarg-Trudinger} the general result can be easily deduced.

Through this paper we always take $0<\sigma<1$.

%%%%%%%%%%%%%%%%%%%%%%%%%%%%%%%%%%%%%%%%%%%%%%%%%%%%%%
\section{Fractional operators and extension problem}\label{Section:Extension}
%%%%%%%%%%%%%%%%%%%%%%%%%%%%%%%%%%%%%%%%%%%%%%%%%%%%%%

Along this paper all the operators will verify the following

\

\noindent\textbf{General assumption.} \textit{By $L=L_x$ we denote a
nonnegative self-adjoint second order partial differential operator
with dense domain $\Dom(L)\subset L^2(\Omega,d\eta)\equiv L^2(\Omega)$.
Here $\Omega$ is an open subset of $\Real^n$, $n\geq1$, and $d\eta$
is a positive measure on $\Omega$. The operator $L$ acts in the variables $x\in\Real^n$.}

\

The Spectral Theorem can be applied to an operator $L$ as in the general assumption, see \cite[Chapter~13]{Rudin}. Given a real measurable function $\phi$ on $[0,\infty)$, the operator $\phi(L)$ is defined as $\phi(L)=\int_0^\infty\phi(\lambda)\,dE(\lambda)$,
where $E$ is the unique resolution of the identity of $L$. The domain $\Dom(\phi(L))$ of $\phi(L)$ is the set of functions $f\in L^2(\Omega)$ such that $\int_0^\infty\abs{\phi(\lambda)}^2\,dE_{f,f}(\lambda)<\infty$.

In this paper we are going to use:

\begin{itemize}
    \item The heat-diffusion semigroup generated by $L$, defined as $\phi(L)=e^{-tL}$, $t\geq0$. For $f\in L^2(\Omega)$, we have that $v=e^{-tL}f$ solves the evolution equation $v_t=-Lv$, for $t>0$. Moreover, $\|e^{-tL}f\|_{L^2(\Omega)}\leq\|f\|_{L^2(\Omega)}$, for all $t\geq0$, and $e^{-tL}f\to f$ in $L^2(\Omega)$ as $t\to0^+$.
    \item The fractional powers of $L$, given by $\phi(L)=L^\sigma$, with domain $\Dom(L^\sigma)\supset\Dom(L)$. When $f\in\Dom(L^\sigma)$ we have $L^\sigma e^{-tL}f=e^{-tL}L^\sigma f$. If $f\in\Dom(L)$ then $\langle Lf,f\rangle=\|L^{1/2}f\|_{L^2(\Omega)}^2$, where $\langle\cdot,\cdot\rangle$ denotes the inner product in $L^2(\Omega)$. Also, for $f\in\Dom(L)$,
        \begin{equation}\label{fraccionario}
        L^\sigma f(x)=\frac{1}{\Gamma(-\sigma)}\int_0^\infty(e^{-tL}f(x)-f(x))\,\frac{dt}{t^{1+\sigma}},\quad\hbox{in}~L^2(\Omega),
        \end{equation}
        where $\Gamma$ is the Gamma function, see for example \cite[p.~260]{Yosida}.
\end{itemize}

We will usually assume that the heat-diffusion semigroup $e^{-tL}$ is positivity-preserving, that is,
\begin{equation}\label{positivity}
f\geq0~\hbox{on}~\Omega~\hbox{implies}~e^{-tL}f\geq0~\hbox{on}~\Omega,~\hbox{for all}~t>0.
\end{equation}

\begin{rem}[Maximum and comparison principle for $L^\sigma$]\label{Rem:comparison}
Let $L$ be as in the general assumption. Under the additional hypothesis \eqref{positivity}, the following comparison principle holds. If $f,g\in\Dom(L)$, $f\ge g$ in
$\Omega$ and $f(x_0)=g(x_0)$ at a point $x_0\in\Omega$, then
$L^\sigma f(x_0)\le L^\sigma g(x_0)$. This comparison principle is a direct consequence of the maximum principle: if $f\in\Dom(L)$,
$f\geq0$, $f(x_0)=0$, then $L^\sigma f(x_0)\leq0$ (for the proof just observe in \eqref{fraccionario} that $\Gamma(-\sigma)<0$ and $e^{-tL}f(x_0)\geq0$).
\end{rem}

\begin{thm}[Extension problem {\cite[Theorem~1.1]{Stinga-Torrea-CPDE}}]\label{Thm:Extension general}
Let $L$ be as in the general assumptions and $f\in\Dom(L^\sigma)$. Let $u$ be defined as
\begin{equation}\label{u with L}
\begin{aligned}
    u(x,y) &:= \frac{y^{2\sigma}}{4^\sigma\Gamma(\sigma)}\int_0^\infty e^{-tL}f(x)e^{-\frac{y^2}{4t}}\,\frac{dt}{t^{1+\sigma}} \\
     &= \frac{1}{\Gamma(\sigma)}\int_0^\infty e^{-tL}(L^\sigma f)(x)e^{-\frac{y^2}{4t}}\,\frac{dt}{t^{1-\sigma}},
\end{aligned}
\end{equation}
for $x\in\Omega$, $y>0$. Then $u\in C^\infty((0,\infty):\Dom(L))\cap
C([0,\infty):L^2(\Omega))$ and it satisfies the extension problem
\begin{equation}\label{equation}
\begin{cases}
    -L_xu+\frac{1-2\sigma}{y}\,u_y+u_{yy}=0, & x\in\Omega,~y>0, \\
    u(x,0)=f(x), & x\in\Omega.
\end{cases}
\end{equation}
In addition, for $c_\sigma=\frac{4^{\sigma-1/2}\Gamma(\sigma)}{\Gamma(1-\sigma)}>0$,
\begin{equation}\label{condition of u}
-c_\sigma\lim_{y\to0^+}y^{1-2\sigma}u_y(x,y)=L^\sigma f(x).
\end{equation}
\end{thm}

We must clarify in which sense the identities in Theorem
\ref{Thm:Extension general} are taken. The first equality in \eqref{u with L} means that for any $g\in
L^2(\Omega)$,
$$\langle u(\cdot,y),g(\cdot)\rangle=\frac{y^{2\sigma}}{4^\sigma\Gamma(\sigma)}\int_0^\infty\langle e^{-tL}f,g\rangle e^{-\frac{y^2}{4t}}\,\frac{dt}{t^{1+\sigma}},\quad y>0,$$
and similarly for the second one. Also \eqref{equation} in general means that $\langle\frac{1-2\sigma}{y}u_y(\cdot,y)+u_{yy}(\cdot,y),g(\cdot)\rangle=\langle Lu(\cdot,y),g(\cdot)\rangle$, for all $y>0$,
with $\langle u(\cdot,y),g(\cdot)\rangle\to\langle f,g\rangle$, as
$y\to0^+$, and analogously for \eqref{condition of u}. By the second identity of \eqref{u with L}, a change of variables and dominated convergence, we have
\begin{equation}\label{L2 limit}
\begin{aligned}
    \limsup_{y\to0^+}\|y^{1-2\sigma}u_y(x,y)\|_{L^2(\Omega')}^2 &\leq \frac{4^{1/2-\sigma}}{\Gamma(\sigma)}\limsup_{y\to0^+}\int_0^\infty\|e^{-\frac{y^2}{4s}L}(L^\sigma f)\|^2_{L^2(\Omega')}e^{-s}\,\frac{ds}{s^{\sigma}} \\
     &= c_\sigma^{-1}\|L^\sigma f\|_{L^2(\Omega')},\quad\hbox{for any measurable set }\Omega'\subseteq\Omega.
\end{aligned}
\end{equation}

%%%%%%%%%%%%%%%%%%%%%%%%%%%%%%%%%%%%%%%%%%%%%%%%%%%%%%
\section{Harnack's inequality for fractional Schr\"odinger operators}\label{Section:Reflection}
%%%%%%%%%%%%%%%%%%%%%%%%%%%%%%%%%%%%%%%%%%%%%%%%%%%%%%

In this section we consider a uniformly elliptic Schr\"odinger operator of the form
$$\L=-\dive(a(x)\nabla)+V,\quad\hbox{on}~\Omega\subseteq\Real^n.$$
Here $a=(a^{ij})$ is a symmetric matrix of real-valued measurable coefficients such that $\mu^{-1}|\xi|^2\le a(x)\xi\cdot\xi\le\mu|\xi|^2$, for some constant $\mu>0$, for almost every $x\in\Omega$ and for all $\xi\in\R^n$. The potential $V$ is a locally bounded function on $\Omega$. Here $\Omega$ can be an unbounded set. We assume that $\L$ satisfies the general assumption at the beginning of Section \ref{Section:Extension}, with $d\eta(x)=dx$, the Lebesgue measure. The domain of $\L$ is $\Dom(\L)=W_0^{1,2}(\Omega)\cap L^2(\Omega,V(x)\,dx)$. The Sobolev space $W^{1,2}_0(\Omega)$ is the completion of $C_c^\infty(\Omega)$ under the norm $\|f\|_{W^{1,2}(\Omega)}^2=\|f\|_{L^2(\Omega)}^2+\|\nabla f\|_{L^2(\Omega)}^2$. Note that $\Dom(\L)$ is dense in $L^2(\Omega)$. For $f\in\Dom(\L)$,
$$\langle\L f,g\rangle=\int_\Omega(a(x)\nabla f\cdot\nabla g+V(x)fg)\,dx,\quad g\in W^{1,2}_0(\Omega)\cap L^2(\Omega,V(x)\,dx).$$

\begin{thm}[Reflection extension]\label{Thm:Reflection}
Fix a ball $B_R(x_0)\subset\Omega$, $x_0\in\Omega$, $R>0$. Let $u:\Omega\times[0,R)\to\Real$ be a solution of the extension equation in \eqref{equation} with $L=\L$ in $B_R(x_0)\times(0,R)$. Define the
reflection of $u$ to $\Omega\times(-R,R)$ by
$\tilde{u}(x,y)=u(x,|y|)$, $x\in\Omega$, $y\in(-R,R)$. Suppose that
\begin{itemize}
    \item[(I)] $\lim_{y\to0^+}\|y^{1-2\sigma}u_y(x,y)\|_{L^2(B_R(x_0),dx)}=0;$ and
    \item[(II)] $\norm{\nabla_x u(x,y)}_{L^2(B_R(x_0),dx)}$ remains bounded as $y\to0^+$.
\end{itemize}
Then $\tilde{u}$ verifies the degenerate Schr\"odinger equation
\begin{equation}\label{weak equation g}
\dive(|y|^{1-2\sigma}b(x)\nabla\tilde{u})-|y|^{1-2\sigma}V(x)\tilde{u}=0,
\end{equation}
in the weak sense in $\tilde{B}:=\set{(x,y)\in\R^{n+1}:|x-x_0|^2+y^2<R^2}$,
where the matrix of coefficients $b=(b^{ij})$ is given by
$b^{ij}=a^{ij}$, $b^{n+1,j}=b^{i,n+1}=0$, $1\leq i,j\leq n$, and
$b^{n+1,n+1}=1$.
\end{thm}

\begin{proof}
Let $\varphi\in C_c^\infty(\tilde{B})$. Take any $0<\delta<R$. Since $u$ is
a solution of the extension equation in \eqref{equation} for $\L$, for any fixed
$y\in(\delta,R)$, we have
$$\int_{B_R(x_0)}(a(x)\nabla_xu\cdot\nabla_x\varphi+V(x)u\varphi)\,dx= \int_{B_R(x_0)}|y|^{2\sigma-1}\partial_y(|y|^{1-2\sigma}u_y)\varphi\,dx.$$
Recall that we are assuming that $u\in C^\infty((0,R):\Dom(\L))$. By integrating the last identity in $y$, applying Fubini's theorem and integration by parts,
\begin{multline*}
\int_\delta^R|y|^{1-2\sigma}\int_{B_R(x_0)}(a(x)\nabla_xu\cdot\nabla_x\varphi+V(x)u\varphi)\,dx\,dy \\
=-\int_{B_R(x_0)}\delta^{1-2\sigma}u_y(x,\delta) \varphi(x,\delta)\,dx-\int_{B_R(x_0)}\int_\delta^R|y|^{1-2\sigma}u_y(x,y)\varphi_y(x,y)\,dy\,dx.
\end{multline*}
From here we get
\begin{multline}\label{passing to x}
\int_{B_R(x_0)\times\set{|y|\geq\delta}}(b(x)\nabla\tilde{u}\cdot\nabla\varphi+
V(x)\tilde{u}\varphi)|y|^{1-2\sigma}\,dx\,dy \\
=\int_{B_R(x_0)}\delta^{1-2\sigma}u_y(x,\delta)\varphi(x,-\delta)\,dx- \int_{B_R(x_0)}\delta^{1-2\sigma}u_y(x,\delta)\varphi(x,\delta)\,dx.
\end{multline}

We are ready to prove that $\tilde{u}$ is a weak solution of \eqref{weak equation g} in $\tilde{B}$. We
have to check that
$$I:=\int_{\tilde{B}}(b(x)\nabla\tilde{u}\cdot\nabla\varphi+V(x)\tilde{u}\varphi)|y|^{1-2\sigma}\,dx\,dy=0.$$
By using \eqref{passing to x},
\begin{equation*}
\begin{aligned}
    I &= \left(\int_{\tilde{B}\cap\set{|y|\ge\delta}}~+\int_{\tilde{B}\cap\set{|y|<\delta}}~\right)\,dx\,dy \\
     &= \int_{B_R(x_0)}\delta^{1-2\sigma}u_y(x,\delta)\varphi(x,-\delta)\,dx- \int_{B_R(x_0)}\delta^{1-2\sigma}u_y(x,\delta)\varphi(x,\delta)\,dx \\
     &\quad+\int_{\tilde{B}\cap\set{|y|<\delta}}b^{ij}\nabla\tilde{u}\cdot\nabla\varphi|y|^{1-2\sigma}\,dx\,dy+ \int_{\tilde{B}\cap\set{|y|<\delta}}V(x)\tilde{u}\varphi|y|^{1-2\sigma}\,dx\,dy.
\end{aligned}
\end{equation*}
As $\delta\to0^+$, the first and second terms above tend to zero because of
(I). Also the fourth term goes to zero because $V(x)\tilde{u}|y|^{1-2\sigma}\in L^1_{\mathrm{loc}}$. Since
$\norm{\nabla_x u(x,y)}_{L^2(B_R(x_0),dx)}$ remains bounded as
$y\to0^+$, for any small $\delta>0$ there exists a constant
$c>0$ such that if $|y|<\delta$ then
$\norm{\nabla_xu(x,y)}_{L^2(B_R(x_0),dx)}\le c$. This property and
(I) imply that the third term above tends to zero as $\delta\to0^+$.
\end{proof}

\begin{thm}[Harnack's inequality for $\L^\sigma$]\label{Thm:Harnack}
Let $\L$ be as above. Assume that the heat-diffusion semigroup $e^{-t\L}$ is positivity-preserving, see \eqref{positivity}. Let $f\in\Dom(\L)$ be a nonnegative function such that $\L^\sigma f=0$ in $L^2(B_R(x_0),dx)$ for some ball $B_R(x_0)\subset\Omega$. Suppose that $\|\nabla_xu(x,y)\|_{L^2(B_R(x_0),dx)}$ remains bounded as $y\to0^+$, where $u$ is a solution to the extension problem \eqref{equation} for $\L$ and $f$. There exist constants $R_0<R$ and $C$ depending only on $n$, $\sigma$, $\mu$, and $V$, but not on $f$, such that,
$$\sup_{B_r}f\le C \inf_{B_r}f,$$
for any ball $B_r$ with $B_{8r}\subset B_R(x_0)$ and $0<r\leq R_0$. Moreover, $f$ is continuous in $B_R(x_0)$.
\end{thm}

In order to prove Theorem \ref{Thm:Harnack} we use Theorem \ref{Thm:Reflection} and the following version of

\

\noindent\textbf{Guti\'errez's Harnack inequality for degenerate Schr\"odinger equations.}
Consider a degenerate Schr\"odinger equation of the form
\begin{equation}\label{degenerate}
-\dive(\tilde{a}(X)\nabla v)+\tilde{V}(X)v=0,\quad X\in\Real^N,
\end{equation}
where $\tilde{a}=(\tilde{a}^{ij})$ is an $N\times N$ symmetric matrix of real-valued measurable coefficients such that $\lambda^{-1}\omega(X)|\xi|^2\le\tilde{a}(X)\xi\cdot\xi\le\lambda\omega(X)|\xi|^2$, for some $\lambda>0$, for almost every $X\in\Real^N$ and for all $\xi\in\R^N$. The function $\omega$ is an $A_2$ weight. The potential $\tilde{V}$ satisfies $\tilde{V}/\omega\in L_\omega^p$ locally, for some large $p=p_{N,\omega}$. Let $\mathcal{O}$ be any open bounded subset of
$\R^N$. Then there exist positive constants $r_0,\gamma$ depending
only on $\lambda$, $N$, $\omega$, $\mathcal{O}$ and $\tilde{V}$ such that if $v$
is any nonnegative weak solution of \eqref{degenerate} in
$\mathcal{O}$ then for every ball $B_r$ with
$B_{8r}\subset\mathcal{O}$ and $0<r\le r_0$ we have
$$\sup_{B_{r/2}}v\le \gamma\inf_{B_{r/2}}v.$$
As a consequence, $v$ is continuous in $\mathcal{O}$. See \cite{Guti}.

\

\begin{proof}[Proof of Theorem \ref{Thm:Harnack}]
Since $\L^\sigma f=0$ in $L^2(B_R(x_0),dx)$, by \eqref{L2 limit} and the hypothesis on $\nabla_xu$, we see that $u$ satisfies the conditions of Theorem \ref{Thm:Reflection}. Now, equation \eqref{weak equation g} is a degenerate Schr\"odinger equation with $A_2$ weight $\omega(x,y)=|y|^{1-2\sigma}$ and potential $\tilde{V}=|y|^{1-2\sigma}V(x)$ such that $\tilde{V}/\omega\in L^p_\omega$ locally for all $p$ sufficiently large. By Guti\'errez's result just explained above, Harnack's inequality for $\tilde{u}$ holds. By restricting $\tilde{u}$ to $y=0$ we get Harnack's inequality for $f$. Moreover, $\tilde{u}$ is continuous in $B_R(x_0)$ and thus $f$.
\end{proof}

%%%%%%%%%%%%%%%%%%%%%%%%%%%%%%%%%%%%%%%%%%%%%%%%%%%%%%
\subsection{The case of nonnegative potentials}
%%%%%%%%%%%%%%%%%%%%%%%%%%%%%%%%%%%%%%%%%%%%%%%%%%%%%%

Under the additional assumptions that $\Omega$ is a bounded set and that the potential $V$ is a nonnegative function in $\Omega$, we can prove Theorem A for $\L^\sigma$. In this case the domain of $\L$ is $\Dom(\L)=W_0^{1,2}(\Omega)$ and it is known that $e^{-t\L}$ is positivity-preserving, see \cite[Chapter~1]{Davies}. Let $f\in W^{1,2}_0(\Omega)$, $f\geq0$, such that $\L^\sigma f=0$ in $L^2(B_R(x_0),dx)$ for some ball $B_R(x_0)\subset\Omega$, $R>0$. Denote by $u$ the solution of the extension problem for $f$ as in Theorem \ref{Thm:Extension general}. By virtue of Theorem \ref{Thm:Harnack}, to prove Harnack's inequality for $\L^\sigma$ we just have to verify that $u$ satisfies condition (II) of Theorem \ref{Thm:Reflection}. As $f\in W^{1,2}_0(\Omega)$, by the ellipticity condition,
\begin{equation}\label{nabla y L}
\mu^{-1}\|\nabla f\|_{L^2(\Omega,dx)}^2\leq\int_\Omega a(x)\nabla f\cdot\nabla f\,dx\leq\langle\L f,f\rangle=\|\L^{1/2}f\|_{L^2(\Omega,dx)}^2,
\end{equation}
(for the last equality see Section \ref{Section:Extension}). Now, since $u\in C^2((0,\infty):W^{1,2}_0(\Omega))$, $\nabla_xu(x,y)$ is
well defined and belongs to $L^2(\Omega,dx)$ for each $y>0$. We can apply \eqref{u with L}, \eqref{nabla y L} and the properties of the heat-diffusion semigroup $e^{-t\L}$ stated at the beginning of Section \ref{Section:Extension} to get
\begin{align*}
    \|\nabla_xu(x,y)&\|_{L^2(B_R(x_0),dx)} \leq \frac{y^{2\sigma}}{4^\sigma\Gamma(\sigma)}\int_0^\infty\|\nabla e^{-t\L}f\|_{L^2(\Omega,dx)}e^{-\frac{y^2}{4t}}\,\frac{dt}{t^{1+\sigma}} \\
     &\leq \mu^{1/2}\frac{y^{2\sigma}}{4^\sigma\Gamma(\sigma)}\int_0^\infty \|e^{-t\L}\L^{1/2}f\|_{L^2(\Omega,dx)}e^{-\frac{y^2}{4t}}\,\frac{dt}{t^{1+\sigma}} \\
     &\leq \mu^{1/2}\frac{\|\L^{1/2}f\|_{L^2(\Omega,dx)}}{\Gamma(\sigma)}\int_0^\infty\left(\frac{y^2}{4t}\right)^\sigma e^{-\frac{y^2}{4t}}\,\frac{dt}{t}=\mu^{1/2}\|\L^{1/2}f\|_{L^2(\Omega,dx)}.
\end{align*}
Thus $\|\nabla_xu(x,y)\|_{L^2(B_R(x_0),dx)}$ remains bounded as
$y\to0^+$ and (II) in Theorem \ref{Thm:Reflection} is valid. Hence Theorem A is proved for this case. Observe that, in particular, Theorem A is valid for the Laplacian in bounded domains with Dirichlet boundary conditions.

\begin{rem}[Liouville theorem for fractional divergence form elliptic operators]\label{Rem:Liouville}
Let $\Omega=\Real^n$ and $V\equiv0$, that is, $\L=-\dive(a(x)\nabla)$. Take $f\in\Dom(\L)=W^{1,2}(\Real^n)$. The following Liouville theorem is true: If $f\geq0$ on $\Real^n$ and $\L^\sigma f=0$ in $L^2(\Real^n)$ then $f$ must be a constant function. Indeed, for this $f$, the reflection $\tilde{u}$ of $u$ is a nonnegative weak solution of \eqref{weak equation g} with $V\equiv0$ in $\Real^{n+1}$, so $\tilde{u}$ is constant and therefore $f$ is a constant function. Here we have applied the Liouville theorem for degenerate elliptic equations in divergence form with $A_2$ weights, which is a simple consequence of Harnack's inequality of \cite{Fabes-Kenig-Serapioni}.
\end{rem}

\begin{rem}
Since our method is based on Guti\'errez's result \cite{Guti}, we are not able to get the exact dependence on $\sigma$ of the constant $C$ in Harnack's inequality of Theorem \ref{Thm:Harnack}.
\end{rem}

%%%%%%%%%%%%%%%%%%%%%%%%%%%%%%%%%%%%%%%%%%%%%%%%%%%%%%
\section{Transference method for Harnack's inequality}\label{Section:Transference}
%%%%%%%%%%%%%%%%%%%%%%%%%%%%%%%%%%%%%%%%%%%%%%%%%%%%%%

In this section we assume that $L$ satisfies the general assumptions of Section \ref{Section:Extension}. We explain in detail a general method to transfer Harnack's inequality from $L^\sigma$ to another operator $\bar{L}^\sigma$ related to $L$. This method will be useful when considering differential operators arising in classical orthogonal expansions and also for the Bessel operator.

Firstly, by a change of measure, we have the following trivial result.

\begin{lem}\label{Lem:change measure}
Let $M(x)\in C^\infty(\Omega)$ be a positive function. Define the
isometry operator $U$ from $L^2(\Omega,M(x)^2d\eta(x))$ into
$L^2(\Omega,d\eta(x))$ as $(Uf)(x)=M(x)f(x)$. Then if
$\{\varphi_k\}_{k\in\N_0^n}$ is an orthonormal system in
$L^2(\Omega,M(x)^2d\eta(x))$ then $\{U\varphi_k\}_{k\in\N_0^n}$ is
also an orthonormal system in $L^2(\Omega,d\eta(x))$.
\end{lem}

Next we set up the notation for the change of variables.

\begin{defn}[Change of variables]\label{change of variables}
Let $h:\Omega\to\bar{\Omega}\subseteq\Real^n$ be a one-to-one $C^\infty$ transformation on $\Omega$. Denote the Jacobian of the
inverse map $h^{-1}:\bar{\Omega}\to\Omega$ by $\abs{J_{h^{-1}}}$.
We define the change of variables operator $W$ from
$L^2(\bar{\Omega},M(h^{-1}(\bar{x}))^2\abs{J_{h^{-1}}}d\eta(\bar{x}))$
into $L^2(\Omega,M(x)^2d\eta(x))$, where $d\eta(x)=\eta(x)\,dx$ for some positive density $\eta$, as
$$(Wf)(x)=f(h(x)),\quad x\in\Omega.$$
\end{defn}

Now we are in position to describe the transference method. By using the definition above and Lemma \ref{Lem:change measure} we construct a new differential operator. This new operator will be nonnegative and self-adjoint in $L^2(\bar{\Omega},d\bar{\eta}(\bar{x}))$, where $\bar{\Omega}=h(\Omega)$ and $d\bar{\eta}(\bar{x}):=M(h^{-1}(\bar{x}))^2\abs{J_{h^{-1}}}d\eta(\bar{x})$. Let
$$\bar{L}:=(U\circ W)^{-1}\circ L\circ(U\circ W).$$
If $E$ is the resolution of the identity of $L$ then the resolution of the identity $\bar{E}$ of $(U\circ W)\circ\bar{L}$ verifies
$$d\bar{E}_{f,g}(\lambda)=dE_{(U\circ W)f,(U\circ W)g}(\lambda),\quad f,g\in L^2(\bar{\Omega},d\bar{\eta}).$$
Therefore if $f\in\Dom(\bar{L}^\sigma)$ then we see that the fractional powers of $\bar{L}$ satisfy
$$\bar{L}^\sigma f=(U\circ W)^{-1}\circ L^\sigma\circ(U\circ W)f.$$

\begin{lem}[Transference method]\label{transference}
If Theorem A for $L^\sigma$ is true, then the analogous statement for $\bar{L}^\sigma$ is also true.
\end{lem}

\begin{proof}
Let $f\in\Dom(\bar{L}^\sigma)$, $f\geq0$, such that $\bar{L}^\sigma f=0$ in $L^2(\bar{\mathcal{O}},d\bar{\eta})$, for some open set $\bar{\mathcal{O}}\subset\bar{\Omega}$. Take a compact set $\bar{K}\subset\bar{\mathcal{O}}$. We want to see that there is a constant $C$ depending on $\bar{K}$ and $\bar{L}^\sigma$ such that
\begin{equation}\label{Harnack A}
\sup_{\bar{K}}f\leq C\inf_{\bar{K}}f.
\end{equation}
Observe that, by the definition of $d\bar{\eta}$ and since $d\eta(x)=\eta(x)\,dx$,
$$\int_{h^{-1}(\bar{\mathcal{O}})}|L^\sigma\circ(U\circ W)f(x)|^2\,d\eta(x)=\int_{\bar{\mathcal{O}}}|\bar{L}^\sigma f(\bar{x})|^2\,d\bar{\eta}(\bar{x})=0,$$
and $(U\circ W)f\in\Dom(L)$ is nonnegative. By the assumption on $L^\sigma$, there exists $C$ depending on $h^{-1}(\bar{K})$ and $L^\sigma$ such that
$$\sup_{h^{-1}(\bar{K})}(U\circ W)f\le C\inf_{h^{-1}(\bar{K})}(U\circ W)f,$$
and $(U\circ W)f$ is continuous. In particular, $f$ is continuous. Since $M(x)$ is positive, continuous and bounded in $h^{-1}(\bar{K})$,
$$\sup_{h^{-1}(\bar{K})}Wf\le C'\inf_{h^{-1}(\bar{K})}Wf.$$
This in turn implies \eqref{Harnack A} as desired.
\end{proof}

%%%%%%%%%%%%%%%%%%%%%%%%%%%%%%%%%%%%%%%%%%%%%%%%%%%%%%
\section{Classical orthogonal expansions}\label{Section:Orthogonal}
%%%%%%%%%%%%%%%%%%%%%%%%%%%%%%%%%%%%%%%%%%%%%%%%%%%%%%

In this section we consider operators $L$ (as in the general assumptions of Section \ref{Section:Extension}) for which there exists a family $\set{\varphi_k}_{k\in\mathbb{N}^n_0}$ of eigenfunctions of $L$, with associated nonnegative eigenvalues $\set{\lambda_k}_{k\in\mathbb{N}^n_0}$, namely, $L\varphi_k(x)=\lambda_k\varphi_k(x)$, such that $\{\varphi_k\}$ is an orthonormal basis of $L^2(\Omega,d\eta)$. In all our examples, the eigenvalues will satisfy the following: there exists a constant $c\geq1$ such that $\lambda_k\sim|k|^c$, for any $k=(k_1,\ldots,k_n)\in \N^n_0$, $|k|=k_1+\cdots+k_n$. We also suppose that the eigenfunctions $\varphi_k$ are in $C^2(\Omega)$ and that their derivatives satisfy the following local estimate. For any compact subset $K\subset\Omega$ and any multi-index $\beta\in\mathbb{N}^n_0$, $\abs{\beta}\leq2$, there exist $\varepsilon=\varepsilon_{K,\beta}\geq0$ and a constant $C=C_{K,\beta}$ such that
\begin{equation}\label{derivatives}
\|D^\beta\varphi_k\|_{L^\infty(K,d\eta)}\le C\abs{k}^\varepsilon,
\end{equation}
for any $k\in\N^n_0$. For $f\in L^2(\Omega,d\eta)$ the heat-diffusion semigroup can be written as $e^{-tL}f(x)=\sum_{\abs{k}=0}^\infty e^{-t\lambda_k}c_k\varphi_k(x)$. For $0<\sigma<1$, the domain of $L^\sigma$ is given as $\Dom(L^\sigma)=\{f\in L^2(\Omega,d\eta):\sum_{\abs{k}=0}^\infty\lambda_k^{2\sigma}|c_k|^2<\infty\}$, where $c_k$ denotes the Fourier coefficient of $f$ in the basis $\varphi_k$: $c_k=\langle f,\varphi_k\rangle=\int_\Omega f\varphi_k\,d\eta$. Given $f\in\Dom(L^\sigma)$ we have $L^\sigma f(x)=\sum_{\abs{k}=0}^\infty\lambda_k^\sigma c_k\varphi_k(x)$.

Under these assumptions we can show that the solution $u$ of the extension problem is classical. To this end, let $K$ be any compact subset of $\Omega$. First we show that the series that defines $e^{-tL}f(x)$ is uniformly convergent in $K\times(0,T)$, for every $T>0$. Indeed, by applying that $\lambda_k\sim|k|^c$, estimate \eqref{derivatives}, the inequality $s^\rho e^{-s}\leq C_\rho e^{-s/2}$ (valid for $s,\rho>0$ and some constant $C_\rho>0$) and Cauchy-Schwartz's inequality,
\begin{align*}
    |e^{-tL}&f(x)| \leq \sum_{\abs{k}\geq0}|e^{-t\lambda_k}c_k\varphi_k(x)|\le \frac{C}{t^{\varepsilon/c}}\sum_{\abs{k}\geq0}(t^{\varepsilon/c}\abs{k}^\varepsilon)e^{-Ct\abs{k}^c}|c_k| \\
     &\le\frac{C}{t^{\varepsilon/c}}\left(\sum_{\abs{k}\geq0} e^{-2Ct\abs{k}^c}\right)^{1/2}\left(\sum_{\abs{k}\geq0}c_k^2\right)^{1/2}\leq \frac{C}{t^{\varepsilon/c}}\left(\sum_{j\geq0}j^ne^{-2Ctj^c}\right)^{1/2}\|f\|_{L^2(\Omega,d\eta)} \\
     &\leq \frac{C}{t^{\frac{\varepsilon+n}{c}}}\left(\sum_{j\geq0}e^{-C'tj^c}\right)^{1/2}\|f\|_{L^2(\Omega,d\eta)}\le \frac{C}{t^{\frac{\varepsilon+n+1/2}{c}}}\norm{f}_{L^2(\Omega,d\eta)},\quad x\in K,
\end{align*}
and the uniform convergence follows. As a consequence, $u$ in \eqref{u with L} is well defined, for by the estimate above, for any $x\in K$ and $y>0$,
$$\int_0^\infty|e^{-tL}f(x)e^{-\frac{y^2}{4t}}|\,\frac{dt}{t^{1+\sigma}}\leq C\norm{f}_{L^2(\Omega,d\eta)}\int_0^\infty\frac{e^{-\frac{y^2}{4t}}}{t^{\frac{\varepsilon+n+1/2}{c}}}\,\frac{dt}{t^{1+\sigma}}\le F(y),$$
for some function $F=F(y)$. This estimate also implies that in the first identity of \eqref{u with L} we can interchange the integration in $t$ with the summation that defines $e^{-tL}f(x)$ to get
\begin{equation}\label{equ sum of u}
u(x,y)=\frac{y^{2\sigma}}{4^\sigma\Gamma(\sigma)}\sum_{\abs{k}\geq0}c_k\varphi_k(x)\int_0^\infty e^{-t\lambda_k}e^{-\frac{y^2}{4t}}\,\frac{dt}{t^{1+\sigma}}.
\end{equation}
By using \eqref{derivatives} and the same arguments as above, it is easy to see that this series
defines a function in $C^2(\Omega)\cap C^1(0,\infty)$. Moreover, since each term of the series in \eqref{equ sum of u} satisfies equation \eqref{equation} in the classical sense, we readily see that $u$ is a classical solution to \eqref{equation}.

Next we will present the concrete applications.

We will take advantage of well-known formulas, see for
instance \cite{AbuMST, AbuTorrea}, to apply our transference method
to get Harnack's inequality for operators of classical orthogonal
expansions which are not of the form considered in Section
\ref{Section:Reflection}. A remarkable advantage of the transference method is that we do not need to check that the semigroup $e^{-tL}$ is positivity-preserving.

%%%%%%%%%%%%%%%%%%%%%%%%%%%%%%%%%%%%%%%%%%%%%%%%%%%%%%
\subsection{Ornstein-Uhlenbeck operator and harmonic oscillator}
%%%%%%%%%%%%%%%%%%%%%%%%%%%%%%%%%%%%%%%%%%%%%%%%%%%%%%

In \cite{GutiRiesz}, Guti\'errez dealt with the Ornstein-Uhlenbeck operator
$$\mathbf{O}_\mathbf{B}=-\Delta+2\mathbf Bx\cdot\nabla,$$
where $\mathbf{B}$ is an $n\times n$ positive definite symmetric
matrix. The operator $\textbf{O}_\mathbf{B}$ is positive and
symmetric in $L^2(\R^n,d\gamma_\mathbf{B}(x))$, where
$d\gamma_\mathbf{B}(x)=(\det
\mathbf{B})^{n/2}\pi^{-n/2}e^{-\mathbf{B}x\cdot x}dx$ is the $\mathbf{B}$-Gaussian measure. Let us consider the eigenvalue problem $\mathbf{O}_\mathbf{B} w = \lambda w$, with boundary conditions $w(x) = O(|x|^k)$, for some $k \ge 0$ as $|x| \rightarrow \infty$. Firstly, let us assume that the matrix $\mathbf{B}$ is diagonal, which
means that
\begin{displaymath} \mathbf{B}=D=\left(\begin{matrix}
  d_1&0 &\cdots&0  \\
  0&d_2 &\cdots&0  \\
  \vdots&\vdots &\ddots&\vdots\\
  0&0&\cdots&d_n
 \end{matrix}\right)
 \end{displaymath}
with $d_i>0$ for $1\le i\le n$. It is not difficult to see that in
this case the eigenfunctions $w$ are the multidimensional Hermite
polynomials defined by $H_k^D(x)=H_{k_1}(\sqrt{d_1}x_1)\cdots
H_{k_n}(\sqrt{d_n}x_n)$, $k\in\mathbb{N}_0^n$, with eigenvalues
$2(k\cdot d)$, $d=(d_1,\ldots, d_n)$, where $H_{k_i}$ is the
one-dimensional Hermite polynomial of degree $k_i$,  see
\cite{GutiRiesz}. For the general case, since $\mathbf{B}$ is a
positive definite symmetric matrix, there exists an orthogonal
matrix $A$ such that $A\mathbf{B}A^t=D,$
 where $A^t$ is the transpose of
$A$. Then the eigenfunctions become $H^{\textbf{B}}_k(x)=H^D_k(Ax)$.

Let us also consider the harmonic oscillator
$$\H_D=-\Delta+|Dx|^2,$$
where $D$ is a matrix as above, with zero boundary condition at infinity. Under these assumptions $\H_D$ is positive and symmetric in $L^2(\Real^n,dx)$. It is well known that the multidimensional Hermite functions $h_k^D(x)=(\det D)^{n/4}\pi^{-n/4}e^{-\frac{Dx\cdot x}{2}}H_k^D(x)$, are the eigenfunctions of $\H_D$ and $\H_D h_k^D=\left(2 (k\cdot d)+\sum_{i=1}^n d_i\right)h_k^D$. The Hermite functions form an orthonormal basis of $L^2(\Real^n,dx)$.

Observe that we may also consider
$$\H_D-\sum_{i=1}^nd_i,$$
since it has the same eigenfunctions as $\H_D$ with eigenvalues $2(k\cdot d)\geq0$. We can also put a more general matrix $\mathbf{B}$ in the place of $D$; we will prove Harnack's inequality for it by using the transference method.

%%%%%%%%%%%%%%%%%%%%%%%%%%%%%%%%%%%%%%%%%%%%%%%%%%%%%%
\subsubsection{Proof of Harnack's inequality for $(\H_D)^\sigma$}
%%%%%%%%%%%%%%%%%%%%%%%%%%%%%%%%%%%%%%%%%%%%%%%%%%%%%%

To show Harnack's inequality for $(\H_D)^\sigma$ we have to check that all the conditions of Theorem \ref{Thm:Harnack} hold.

The potential here is $V(x)=|Dx|^2$, which is a locally bounded function on $\Real^n$.

By Mehler's formula \cite{GutiRiesz, Szego, Thangavelu}, $e^{-t\H_D}$ is positivity-preserving.

In \cite{Thangavelu}, it is shown that there exists $C$ such that $\|h_k^D\|_{L^\infty(\Real^n,dx)}\leq C$ for
all $k$. Using the relation
$$2\partial_{x_i}h_k^D(x)=\sqrt{d_i}\left((2k_i)^{1/2}h_{k-e_i}^D(x)-(2k_i+2)^{1/2}h_{k+e_i}^D(x)\right),$$
where $e_i$ is the $i$-th coordinate vector in $\mathbb{N}_0^n$, we
see that \eqref{derivatives} is valid for $h_k^D(x)$. Therefore the
solution $u$ to the extension problem given in \eqref{u with L} for $\H_D$ is a classical
solution.

Let $f\in\Dom(\H_D)$, $f\geq0$, such that $(\H_D)^\sigma f=0$ in $L^2(B_R(x_0),dx)$. We have to verify that $\norm{\nabla_xu(x,y)}_{L^2(B_R(x_0),dx)}$ remains bounded as $y\to0^+$. In fact, we will have
$\norm{\nabla_xu(x,0)}_{L^2(B_R(x_0),dx)}=\norm{\nabla_xf(x)}_{L^2(B_R(x_0),dx)}$. Indeed, as we can write $f=\sum_{|k|=0}^\infty c_k
h_k^D$, by \eqref{equ sum of u} and the identity for the derivatives of the Hermite functions $h_k^D$ given above,
\begin{equation}\label{equ11}
\begin{aligned}
&\left(\partial_{x_i}+\sqrt{d_i}x_i\right)\left(u(x,y)-f(x)\right)\\&=\sum_k
c_k\sqrt{d_i}(2k_i)^{1/2}h_{k-e_i}^D(x)\left(\frac{y^{2\sigma}}{4^\sigma\Gamma(\sigma)}\int_0^\infty
e^{-t(k\cdot d+\sum_{l=1}^n
d_l)}e^{-\frac{y^2}{4t}}\frac{dt}{t^{1+\sigma}}-1\right).
\end{aligned}
\end{equation}
Observe that the term in parenthesis above is uniformly bounded in
$y$ and, since
$$\frac{y^{2\sigma}}{4^\sigma\Gamma(\sigma)}\int_0^\infty e^{-t(k\cdot d+\sum_{l=1}^n
d_l)}e^{-\frac{y^2}{4t}}\,\frac{dt}{t^{1+\sigma}}=\frac{1}{\Gamma(\sigma)}
\int_0^\infty e^{-\frac{y^2}{4w}(k\cdot d+\sum_{l=1}^n
d_l)}e^{-w}\,\frac{dw}{w^{1-\sigma}},$$ we readily see that it converges to $0$ when $y\to 0^+$. Moreover, as
$f\in\Dom(\H_D)$,
\begin{equation*}
\norm{\sum_{\abs{k}=0}^\infty
c_k\sqrt{d_i}(2k_i)^{1/2}h_{k-e_i}^D(x)}_{L^2(\Real^n,dx)}=\left(2\sum_{\abs{k}=0}^\infty
c_k^2{k_id_i}\right)^{1/2}<\infty.
\end{equation*}
Hence, by dominated convergence in \eqref{equ11}, we get that $\left(\partial_{x_i}+\sqrt{d_i}x_i\right)u(x,y)\to
\left(\partial_{x_i}+\sqrt{d_i}x_i\right)f(x)$ in $L^2(\Real^n,dx)$ as $y\to0^+$. Since $u(x,y)\to f(x)$ in $L^2(\Real^n,dx)$ and
$\sqrt{d_i}x_i$ is a bounded function in $B_R(x_0)$, we have that $\sqrt{d_i}x_i u(x,y)$ converges to $\sqrt{d_i}x_if(x)$ in
$L^2(B_R(x_0),dx)$ as $y\to 0^+$. Hence
$\nabla_xu(x,y)\to\nabla_xf(x)$, as $y\to0^+$, in
$L^2(B_R(x_0),dx)$.

%%%%%%%%%%%%%%%%%%%%%%%%%%%%%%%%%%%%%%%%%%%%%%%%%%%%%%
\subsubsection{Proof of Harnack's inequality for $\big(\mathbf{O}_D\big)^\sigma$}
%%%%%%%%%%%%%%%%%%%%%%%%%%%%%%%%%%%%%%%%%%%%%%%%%%%%%%

We apply the transference method explained in Section
\ref{Section:Transference}. For this case we take $M(x)=(\det
D)^{n/4}\pi^{-n/4}e^{-\frac{Dx\cdot x}{2}}$ and $h(x)=x$. Clearly
$h_k^D(x)=(U\circ W)H_k^D(x)$ and we have the relation
\begin{equation}\label{relation of Hermite}
\textbf{O}_DH_k^D=(U\circ W)^{-1}\circ\left(\H_D-\sum_{i=1}^n
d_i\right)\circ(U\circ W)H_k^D.
\end{equation}
See also \cite{AbuTorrea}. It can be easily checked, as done for $(\H_D)^\sigma$ above, that the operator $\left(\H_D-\sum_{i=1}^n d_i\right)^\sigma$
satisfies Harnack's inequality. Hence the conclusion for
$(\mathbf{O}_D)^\sigma$ follows from Lemma \ref{transference}.

%%%%%%%%%%%%%%%%%%%%%%%%%%%%%%%%%%%%%%%%%%%%%%%%%%%%%%
\subsubsection{Proof of Harnack's inequality for $\big(\mathbf{O}_{\mathbf{B}}\big)^\sigma$}\label{5.1.algo}
%%%%%%%%%%%%%%%%%%%%%%%%%%%%%%%%%%%%%%%%%%%%%%%%%%%%%%

Consider the change of variables $h(x) = A^tx$ and
call $W$ the corresponding operator as in Definition \ref{change of
variables}. Then it is easy to check that
$$\mathbf{O}_\mathbf{B}(H_k^D\circ h^{-1})(h(x) ) = \mathbf{O}_\mathbf{D}H^{D}_k(x).$$
Then we have $  \mathbf{O}_\mathbf{B} =W^{-1} \circ \mathbf{O}_D\circ W$ and the result follows by the transference method.

%%%%%%%%%%%%%%%%%%%%%%%%%%%%%%%%%%%%%%%%%%%%%%%%%%%%%%
\subsubsection{Proof of Harnack's inequality for $\big(\H_{\mathbf{B}}\big)^\sigma$}
%%%%%%%%%%%%%%%%%%%%%%%%%%%%%%%%%%%%%%%%%%%%%%%%%%%%%%

We observe that parallel to the case of the operator $
\mathbf{O}_\mathbf{B}$ we can get $\H_\mathbf{B}= W^{-1}\circ
\H_D\circ W$ with $W$ as in Subsection \ref{5.1.algo} above and then we get Harnack's inequality
for the operator $(\H_\mathbf{B})^\sigma.$

%%%%%%%%%%%%%%%%%%%%%%%%%%%%%%%%%%%%%%%%%%%%%%%%%%%%%%
\subsection{Laguerre operators}
%%%%%%%%%%%%%%%%%%%%%%%%%%%%%%%%%%%%%%%%%%%%%%%%%%%%%%

We suggest the reader to check \cite{AbuMST, Gutierrez-Incognito-Torrea, Lebedev, Szego, Thangavelu} for the proof of the basics about Laguerre expansions we use here. Let us consider the system of multidimensional Laguerre polynomials $L_k^\alpha(x)$, where $k\in\N_0^n$, $\alpha=(\alpha_1,\cdots,\alpha_n)\in(-1,\infty)^n$ and $x\in(0,\infty)^n$. It is well known that the Laguerre polynomials form a complete orthogonal system in $L^2((0,\infty)^n,d\gamma_\alpha(x))$, where $d\gamma_\alpha(x)=x_1^{\alpha_1}e^{-x_1}\,dx_1\cdots x_n^{\alpha_n}e^{-x_n}\,dx_n$. We denote by $\tilde{L}_k^\alpha$ the orthonormalized Laguerre polynomials. The polynomials $\tilde{L}_k^\alpha$ are eigenfunctions of the Laguerre differential operator
$$\mathbf{L}_\alpha=\sum_{i=1}^n\left(-x_i\frac{\partial^2}{\partial x_i^2}-(\alpha_i+1-x_i)\frac{\partial}{\partial x_i}\right),$$
namely, $\mathbf{L}_\alpha(\tilde{L}_k^\alpha)=|k|\tilde{L}_k^\alpha$. There are several systems of Laguerre functions. We first prove Harnack's inequality for the operator $\mathbf{L}_\alpha^\varphi$ (related to the system $\varphi_k^\alpha$ below) and then we apply the transference method of Section \ref{Section:Transference} to get the result for the remaining systems.

%%%%%%%%%%%%%%%%%%%%%%%%%%%%%%%%%%%%%%%%%%%%%%%%%%%%%%
\subsubsection{Laguerre functions $\varphi_k^\alpha$}
%%%%%%%%%%%%%%%%%%%%%%%%%%%%%%%%%%%%%%%%%%%%%%%%%%%%%%

This multidimensional system in $L^2((0,\infty)^n,d\mu_0(x))$, where $d\mu_0(x)=dx_1\cdots dx_n$, is given as a tensor product $\varphi_k^\alpha(x)=\varphi_{k_1}^{\alpha_1}(x_1)\cdots\varphi_{k_n}^{\alpha_n}(x_n)$, where each factor $\varphi_{k_i}^{\alpha_i}(x_i)=x_i^{\alpha_i}(2x_i)^{1/2}e^{-x_i^2/2}\tilde{L}_{k_i}^{\alpha_i}(x_i^2).$
The functions $\varphi_k^\alpha$ are eigenfunctions of the differential operator
\begin{equation}\label{oper}
\mathbf{L}^\varphi_\alpha=\frac{1}{4}\left(-\Delta+|x|^2\right)+\sum_{i=1}^n\frac{1}{4x_i^2}\left(\alpha_i^2-\frac{1}{4}\right),
\end{equation}
namely,
\begin{equation}\label{eigen of phi}
\mathbf{L}^\varphi_\alpha\varphi_k^\alpha(x)=\sum_{i=0}^n\left(k_i+\frac{\alpha_i+1}{2}\right)\varphi_{k_i}^{\alpha_i}(x_i).
\end{equation}
Clearly, the functions $\varphi_k^\alpha$ are locally bounded in $(0,\infty)^n$. Observe that
\begin{equation}\label{derivatives varphi}
\partial_{x_i}\varphi_k^\alpha(x)=-|k|^{1/2}\varphi_{k-e_i}^{\alpha_i+e_i}(x)- \left(x_i-\frac{1}{x_i}\left(\alpha_i+\frac{1}{2}\right)\right)\varphi_k^\alpha(x).
\end{equation}
Therefore, \eqref{derivatives} holds for this system and we get that the solution $u$ in \eqref{u with L} of the extension problem for $\mathbf{L}_\alpha^\varphi$ is classical. Moreover, it can be easily seen from \cite[p.~102]{Szego} that $e^{-t\mathbf{L}_\alpha^\varphi}$ is positivity-preserving.

Let us prove Theorem A for $(\mathbf{L}_\alpha^\varphi)^\sigma$. We can do this as we did for
$(\H_D)^\sigma$ above by following the reasoning line by line,
but with some modifications as follows. Let
$f\in\Dom(\mathbf{L}^\varphi_\alpha)$, $f\geq0$, such that
$(\mathbf{L}^\varphi_\alpha)^\sigma f=0$ in
$L^2(B_R(x_0),d\mu_0(x))$, and let $u$ be the corresponding solution to the extension problem. By
\eqref{derivatives varphi} and a similar argument for that of
$\H^\sigma$ we can check that
$\norm{\nabla_xu(x,y)}_{L^2(B_R(x_0),d\mu_0(x))}$ converges to
$\norm{\nabla_xf}_{L^2(B_R(x_0),d\mu_0(x))}$, as $y\to0^+$.
Moreover, the potential in \eqref{oper} is locally bounded. Hence, by Theorem \ref{Thm:Harnack}, $f$ satisfies Harnack's inequality and it is continuous.

Note that the same arguments above can be used for $(\mathbf{L}_\alpha^\varphi-\frac{\alpha+1}{2})^\sigma$ instead of $(\mathbf{L}_\alpha^\varphi)^\sigma$, so it also satisfies Theorem A.

%%%%%%%%%%%%%%%%%%%%%%%%%%%%%%%%%%%%%%%%%%%%%%%%%%%%%%
\subsubsection{Laguerre functions $\ell_k^\alpha$}
%%%%%%%%%%%%%%%%%%%%%%%%%%%%%%%%%%%%%%%%%%%%%%%%%%%%%%

The Laguerre functions $\ell_k^\alpha$ are defined as $\ell_k^\alpha(x)=\ell_{k_1}^{\alpha_1}(x_1)\cdots\ell_{k_n}^{\alpha_n}(x_n)$, where $\ell_{k_i}^{\alpha_i}$ are the one-dimensional Laguerre functions $\ell_{k_i}^{\alpha_i}(x_i)=e^{-x_i/2}\tilde{L}_{k_i}^{\alpha_i}(x_i)$. Each $\ell_k^\alpha$ is an
eigenfunction of the differential operator
$$\mathbf{L}^\ell_\alpha=\sum_{i=1}^{n}\left(-x_i\frac{\partial^2}{\partial x_i^2}-(\alpha_i+1)\frac{\partial}{\partial x_i}+\frac{x_i}{4}\right).$$
More explicitly, $\mathbf{L}^\ell_\alpha\ell_k^\alpha=\sum_{i=1}^{n}\left(k_i+\frac{\alpha_i+1}{2}\right)\ell_{k_i}^{\alpha_i}$. For $d\mu_\alpha(x)=x_1^{\alpha_1}\cdots x_n^{\alpha_n}dx$, the operator $\mathbf{L}^\ell_\alpha$ is positive and symmetric in $L^2((0,\infty)^n,d\mu_\alpha(x))$. The system $\{\ell_k^\alpha: k\in \N^n_0\}$ is an orthonormal basis of $L^2((0,\infty)^n,d\mu_\alpha(x))$.

To apply the transference method we set $M(x)=2^{{n}/{2}}x_1^{\alpha_1+1/2}\cdots x_n^{\alpha_n+1/2}$ and $h(x)=(x_1^2,\ldots,x_n^2)$. Then $U\circ W$ is an isometry from $L^2((0,\infty)^n,d\mu_\alpha(x))$ into $L^2((0,\infty)^n,d\mu_0(x))$ and $\mathbf{L}_\alpha^\ell=(U\circ W)^{-1}\circ\mathbf{L}_\alpha^{\varphi}\circ(U\circ W)$, see \cite{AbuMST}.

%%%%%%%%%%%%%%%%%%%%%%%%%%%%%%%%%%%%%%%%%%%%%%%%%%%%%%
\subsubsection{Laguerre functions $\psi_k^\alpha$}
%%%%%%%%%%%%%%%%%%%%%%%%%%%%%%%%%%%%%%%%%%%%%%%%%%%%%%

Consider the Laguerre system $\psi_k^\alpha(x)=\psi_{k_1}^{\alpha_1}(x_1)\cdots\psi_{k_n}^{\alpha_n}(x_n)$, which is orthonormal in $L^2((0,\infty)^n,d\mu_{2\alpha+1}(x))$, where $d\mu_{2\alpha+1}(x)=x_1^{2\alpha_1+1}dx_1\cdots x_n^{2\alpha_n+1}dx_n$ and $\psi_{k_i}^{\alpha_i}$ is the one-dimensional Laguerre function
$\psi_{k_i}^{\alpha_i}(x_i)=\sqrt{2}\ \ell_{k_i}^{\alpha_i}(x_i^2).$
The functions $\psi_k^\alpha$ are eigenfunctions of the operator
$$\mathbf{L}^\psi_\alpha=\frac{1}{4}\left(-\Delta+|x|^2\right)-\sum_{i=1}^n\frac{2\alpha_i+1}{4x_i}\frac{\partial}{\partial x_i}.$$
In fact, $\mathbf{L}^\psi_\alpha(\psi_k^\alpha)=\sum_{i=0}^n\left(k_i+\frac{\alpha_i+1}{2}\right)\psi_{k_i}^{\alpha_i}$.

For the transference method we have to take $M(x)=x_1^{\alpha_1+1/2}\cdots x_n^{\alpha_n+1/2}$ and $h(x)=x$. Then $U\circ W$ is an isometry from $L^2((0,\infty)^n,d\mu_{2\alpha+1}(x))$ into $L^2((0,\infty)^n,d\mu_0(x))$ and $\mathbf{L}_\alpha^\psi=(U\circ W)^{-1}\circ\mathbf{L}_\alpha^{\varphi}\circ(U\circ W)$, see \cite{AbuMST}.

%%%%%%%%%%%%%%%%%%%%%%%%%%%%%%%%%%%%%%%%%%%%%%%%%%%%%%
\subsubsection{Laguerre functions $\L_k^\alpha$}
%%%%%%%%%%%%%%%%%%%%%%%%%%%%%%%%%%%%%%%%%%%%%%%%%%%%%%

The functions $\L_k^\alpha(x)=\L_{k_1}^{\alpha_1}(x_1)\cdots\L_{k_n}^{\alpha_n}(x_n)$ form an orthonormal system in $L^2((0,\infty)^n,d\mu_0(x))$, where $\L_{k_i}^{\alpha_i}$ is the one-dimensional Laguerre function given by $\L_{k_i}^{\alpha_i}(x_i)=x_i^{{\alpha_i}/{2}}\ell_{k_i}^{\alpha_i}(x_i)$. The functions $\L_k^\alpha$ are eigenfunctions of the operator
$$\mathbf{L}^\L_\alpha=\sum_{i=1}^n\left(-x_i\frac{\partial^2}{\partial x_i^2}-\frac{\partial}{\partial x_i}+\frac{x_i}{4}+\frac{\alpha_i^2}{4x_i}\right).$$
In fact, $\mathbf{L}^\L_\alpha(\L_k^\alpha)=\sum_{i=0}^n \left(k_i+\frac{\alpha_i+1}{2}\right)\L_{k_i}^{\alpha_i}$.

Apply the transference method with $M(x)=2^{n/2}x_1^{1/2}\cdots x_n^{1/2}$ and $h(x)=(x_1^2,\ldots,x_n^2)$. Then $U\circ W$ is an isometry from $L^2((0,\infty)^n,d\mu_0(x))$ into itself and $\mathbf{L}_\alpha^\L=(U\circ W)^{-1}\circ\mathbf{L}_\alpha^{\varphi}\circ(U\circ W)$, see \cite{AbuMST}.

%%%%%%%%%%%%%%%%%%%%%%%%%%%%%%%%%%%%%%%%%%%%%%%%%%%%%%
\subsubsection{Laguerre polynomials $\tilde{L}_k^\alpha$}
%%%%%%%%%%%%%%%%%%%%%%%%%%%%%%%%%%%%%%%%%%%%%%%%%%%%%%

Finally consider the Laguerre polynomials operator $\mathbf{L}_\alpha$. Let $M(x)=2^{n/2}e^{-\abs{x}^2/2}x_1^{\alpha_1+1/2}\cdots x_n^{\alpha_n+1/2}$ and $h(x)=(x_1^2,\ldots,x_n^2)$. We have that the operator $U\circ W$ is an isometry from $L^2((0,\infty)^n,d\gamma_\alpha(x))$ into $L^2((0,\infty)^n,d\mu_0(x))$ and $\mathbf{L}_\alpha=(U\circ W)^{-1}\circ(\mathbf{L}^\varphi_\alpha-\tfrac{\alpha+1}{2})\circ(U\circ W)$, see \cite{AbuMST}, so the transference method applies.

%%%%%%%%%%%%%%%%%%%%%%%%%%%%%%%%%%%%%%%%%%%%%%%%%%%%%%
\subsection{Ultraspherical operators}
%%%%%%%%%%%%%%%%%%%%%%%%%%%%%%%%%%%%%%%%%%%%%%%%%%%%%%

Here we restrict ourselves to one-dimensional expansions. We denote the ultraspherical polynomials of type $\lambda>0$ and degree $k\in\mathbb{N}_0$ by $P_k^\lambda(x)$, $x\in(-1,1)$, see \cite{Lebedev, Muck Stein, Szego}. It is well-known that the set of trigonometric polynomials $\{P_k^\lambda(\cos\theta):\theta\in(0,\pi)\}$ forms an orthogonal basis of $L^2((0,\pi),dm_\lambda(\theta))$, where $dm_\lambda(\theta)=\sin^{2\lambda}\theta\,d\theta$. The polynomials $P_k^\lambda(\cos\theta)$ are eigenfunctions of the ultraspherical operator
$$L_\lambda=-{\frac{d^2}{d\theta^2}}-2\lambda\cot\theta\frac{d}{d\theta}+\lambda^2,$$
that is, $L_\lambda P_k^\lambda(\cos\theta)=(k+\lambda)^2P_k^\lambda(\cos\theta)$. We denote by $\tilde{P}_k^\lambda(\cos\theta)$ the orthonormalized polynomials given by $\frac{\Gamma(\lambda)^2(n+\lambda)n!}{2^{1-2\lambda}\pi\Gamma(n+2\lambda)}P_k^\lambda(\cos\theta)$. There exists a constant $A$ such that $|P_k^\lambda(\cos\theta)|\leq Ak^{2\lambda-1}$, see \cite{Muck Stein}. This and Stirling's formula for the Gamma function \cite{Lebedev} imply that there exists $C$ such that $|\tilde{P}_k^\lambda(\theta)|\leq Ck$ for all $k$. A similar estimate holds for the derivatives of $\tilde{P}_k^\lambda$ since $\frac{d}{dx}P_k^\lambda(x)=2\lambda P_{k-1}^{\lambda+1}(x)$, see \cite{Szego}.

The set of orthonormal ultraspherical functions $p_k^\lambda(\theta)=\sin^\lambda\theta\tilde{P}_k^\lambda(\cos\theta)$ is a basis of $L^2((0,\pi),dx)$. The ultraspherical functions are eigenfunctions of the differential operator
$$l_\lambda=-\frac{d^2}{d\theta^2}+\frac{\lambda(\lambda-1)}{\sin^2\theta},$$
namely, $l_\lambda p_k^\lambda(\theta)=(k+\lambda)^2p_k^\lambda(\theta)$. By using the estimates for $\tilde{P}_k^\lambda$ given above, we can easily check that this system satisfies \eqref{derivatives}. Moreover, the heat-diffusion semigroup $e^{-tl_\lambda}$ is positivity-preserving. This last assertion can be deduced directly from the facts that the heat-diffusion semigroup for the ultraspherical polynomials $e^{-tL_\lambda}$ is positivity preserving, see \cite{Bochner}, and $e^{-tl_\lambda}=(U\circ W)\circ(e^{-tL_\lambda})\circ(U\circ W)^{-1}$, see Subsection \ref{ultra} below.

%%%%%%%%%%%%%%%%%%%%%%%%%%%%%%%%%%%%%%%%%%%%%%%%%%%%%%
\subsubsection{Proof of Harnack's inequality for $(l_\lambda)^\sigma$}
%%%%%%%%%%%%%%%%%%%%%%%%%%%%%%%%%%%%%%%%%%%%%%%%%%%%%%

We do this as we did for $(\H_D)^\sigma$ above by following parallel
arguments. Let $f\in\Dom(l_\lambda)$, $f\geq0$, such that
$(l_\lambda)^\sigma f=0$ in $L^2(I,d\theta)$, for some interval
$I\subset(0,\pi)$. Let $u$ be the solution to the extension problem
for $l_\lambda$ and this $f$. By the estimates mentioned above, $u$ is classical.
The potential here is
$V(\theta)=\frac{\lambda(\lambda-1)}{\sin^2\theta}$, which is a locally bounded function.
Observe that $\frac{d}{d\theta}p_k^\lambda(\theta)=-2\lambda
p_{k-1}^{\lambda+1}(\theta)+\lambda\cot\theta p_k^\lambda(\theta)$.
Since $\cot\theta$ is bounded in $I$, by following the same
arguments as those for $(\H_D)^\sigma$, we can get
$\|\frac{\partial}{\partial\theta}u(\theta,y)\|_{L^2(I,d\theta)}\to\|f'(\theta)\|_{L^2(I,d\theta)}$, as $y\to0^+$.
The conclusion follows by Theorem \ref{Thm:Harnack}.

%%%%%%%%%%%%%%%%%%%%%%%%%%%%%%%%%%%%%%%%%%%%%%%%%%%%%%
\subsubsection{Proof of Harnack's inequality for $(L_\lambda)^\sigma$}\label{ultra}
%%%%%%%%%%%%%%%%%%%%%%%%%%%%%%%%%%%%%%%%%%%%%%%%%%%%%%

This is achieved by applying the transference method with $M(\theta)=\sin^{\lambda}\theta$ and $h(\theta)=\theta$. It readily follows that $(L_\lambda)^\sigma=(U\circ W)^{-1}\circ(l_\lambda)^\sigma\circ(U\circ W)$.

%%%%%%%%%%%%%%%%%%%%%%%%%%%%%%%%%%%%%%%%%%%%%%%%%%%%%%
\section{Laplacian and Bessel operators}\label{Section:Laplacian-Bessel}
%%%%%%%%%%%%%%%%%%%%%%%%%%%%%%%%%%%%%%%%%%%%%%%%%%%%%%

In this section we will prove Theorem A for the fractional powers of the Bessel operator. This operator is a generalization of the radial Laplacian. For the sake of completeness and to show how the proof works, we present first the case of the fractional Laplacian on $\Real^n$, for which the more familiar Fourier transform applies.

The main difference with respect to the examples given before is that these operators have a continuous spectrum and the Fourier and Hankel transforms come into play.

%%%%%%%%%%%%%%%%%%%%%%%%%%%%%%%%%%%%%%%%%%%%%%%%%%%%%%
\subsection{The Laplacian on $\Real^n$}
%%%%%%%%%%%%%%%%%%%%%%%%%%%%%%%%%%%%%%%%%%%%%%%%%%%%%%

Consider the fractional Laplacian defined by $\widehat{(-\Delta)^\sigma f}(\xi)=\abs{\xi}^{2\sigma}\widehat{f}(\xi)$, where $\widehat{f}$ denotes the Fourier transform: $\displaystyle\widehat{f}(\xi)\equiv c_\xi(f)=\frac{1}{(2\pi)^{n/2}}\int_{\Real^n}f(x)e^{-ix\cdot\xi}\,dx$, $\xi\in\Real^n$. The eigenfunctions of $-\Delta$, indexed by the continuous parameter
$\xi$, are $\varphi_\xi(x)=e^{-ix\cdot\xi}$, $x\in\R^n$, and
$(-\Delta)\varphi_\xi(x)=|\xi|^2\varphi_\xi(x)$. Note that for any
compact subset $K\subset\R^n$ and any multi-index $\beta\in\N_0^n$,
$|\beta|\le 2$, we have
\begin{equation}\label{deriv of Laplacian}
\norm{D^\beta \varphi_\xi}_{L^\infty(K)}\le|\xi|^{|\beta|}.
\end{equation}
For any $f\in L^2(K, dx),$ the heat semigroup is defined by
$\displaystyle e^{t\Delta}f(x) =\frac{1}{(2\pi)^{n/2}} \int_{\R^n}
e^{-t|\xi|^2}c_\xi(f)\varphi_{-\xi}(x)\,d\xi$. As
\begin{equation}\label{etL}
\abs{e^{t\Delta}f(x)}\leq C\int_{\R^n}\abs{e^{-t|\xi|^2}c_\xi(f)\varphi_{-\xi}(x)}\,d\xi\le
Ct^{-n/4}\norm{f}_{L^2(K,dx)},\quad x\in K,
\end{equation}
the integral that defines $e^{t\Delta}f(x)$ is absolutely convergent in $K\times (0,T)$ with
$T>0$. Moreover, $e^{t\Delta}$ is positivity-preserving in the sense of \eqref{positivity} because it is given by convolution with the Gauss-Weierstrass kernel. Note that, in this spectral language, $\Dom(-\Delta)=\big\{f\in L^2(\R^n,dx):|\xi|^2\hat{f}(\xi)\in L^2(\R^n,dx)\big\}=\big\{f\in L^2(\R^n,dx):D^2f\in L^2(\R^n,dx)\big\}=W^{2,2}(\Real^n)$, the Sobolev space of functions in $L^2(\Real^n)$ with Hessian $D^2f$ in $L^2(\Real^n)$.

Let us show Theorem A for $(-\Delta)^\sigma$. Assume that
$f\in W^{2,2}(\Real^n)$, $f\ge0$ and $(-\Delta)^\sigma f=0$ in
$L^2(B_R,dx)$, for some ball $B_R\subset\R^n$. By Theorem \ref{Thm:Harnack}, we just must check that $\|\nabla_xu(x,y)\|_{L^2(B_R,dx)}$ remains bounded as $y\to0^+$. To that end, observe that for any $x\in B_R$ and $y>0,$ by \eqref{etL},
$$\int_0^\infty|e^{t\Delta}f(x)e^{-\frac{y^2}{4t}}|\,\frac{dt}{t^{1+\sigma}}\leq C\norm{f}_{L^2(B_R,dx)}
\int_0^\infty
t^{-n/4}{e^{-\frac{y^2}{4t}}}\,\frac{dt}{t^{1+\sigma}}\le
F(y),$$ for some function $F(y)$.  This means that we can
interchange integrals in $u$ to get
\begin{equation}\label{u of Lap}
u(x,y)=\frac{y^{2\sigma}}{4^\sigma\Gamma(\sigma)(2\pi)^{n/2}}\int_{\R^n}c_\xi(f)\varphi_{-\xi}(x)\int_0^\infty e^{-t|\xi|^2}e^{-\frac{y^2}{4t}}\,\frac{dt}{t^{1+\sigma}}\,d\xi.
\end{equation}
By \eqref{deriv of Laplacian} and using the same arguments as above,
it is easy to see that this double integral defines a function in
$C^2(B_R\times (0,\infty))$. So in this case $u$ is a classical
solution of \eqref{equation}. By using Plancherel's Theorem and
\eqref{u of Lap} we have
\begin{equation}\label{parti for Lap}
\begin{aligned}
    \|\partial_{x_j}&(u(x,y)-f(x))\|_{L^2(\R^n,dx)}^2 =\norm{\left({\partial_{x_j}(u(x,y)-f(x))}\right)^{\widehat{~}}(\xi)}_{L^2(\R^n,d\xi)}^2\\
     &=\frac{1}{(2\pi)^n}\int_{\R^n}\abs{(-i\xi_j)c_\xi(f)
    \left[\frac{y^{2\sigma}}{4^\sigma\Gamma(\sigma)}\int_0^\infty
    e^{-t|\xi|^2}e^{-\frac{y^2}{4t}}\frac{dt}{t^{1+\sigma}}-1\right]}^2d\xi \\
     &= \frac{1}{(2\pi)^n}\int_{\R^n}\abs{(-i\xi_j)c_\xi(f)\varphi_{-\xi}(x)}^2\left[\frac{y^{2\sigma}}{4^\sigma\Gamma(\sigma)}
     \int_0^\infty\left(e^{-t|\xi|^2}-1\right)e^{-\frac{y^2}{4t}}\frac{dt}{t^{1+\sigma}}\right]^2d\xi.
\end{aligned}
\end{equation}
Observe that the expression in square brackets above is uniformly
bounded in $y$ and it converges to $0$ when $y\to 0^+$. Moreover, as $f\in W^{2,2}(\Real^n)$, $\norm{(-i\xi_j)c_\xi(f)}_{L^2(\R^n,d\xi)}=\|\partial_{x_j}f\|_{L^2(\Real^n,dx)}<\infty$. Hence, by dominated convergence in \eqref{parti for Lap}, $\partial_{x_j}u(x,y)$ converges to $\partial_{x_j}f$ in $L^2(\R^n,dx)$ as $y\to 0^+$. Whence $\nabla_xu(x,y)\to\nabla_xf(x)$
as $y\to 0^+$, in $L^2(B_R,dx)$.

%%%%%%%%%%%%%%%%%%%%%%%%%%%%%%%%%%%%%%%%%%%%%%%%%%%%%%
\subsection{The Bessel operators on $(0,\infty)$}
%%%%%%%%%%%%%%%%%%%%%%%%%%%%%%%%%%%%%%%%%%%%%%%%%%%%%%

Let $\lambda>0$. Let us denote by $\Delta_\lambda$ the Bessel operator
$$\Delta_\lambda=-\frac{d^2}{dx^2}-\frac{2\lambda}{x}\frac{d}{dx},\quad x>0,$$
which is positive and symmetric in $L^2((0,\infty),dm_\lambda(x))$,
where $dm_\lambda(x)=x^{2\lambda}dx$, see \cite{BetanDT, Muck Stein}. If $2\lambda=n-1$, $n\in\N$, then we recover the radial Laplacian on $\Real^n$. Let $J_\nu$ denote the Bessel function of the first kind
with order $\nu$ and let us  define
$\varphi_\xi^\lambda(x)=x^{-\lambda}(\xi x)^{1/2}J_{\lambda-1/2}(\xi
x)$, $x,\xi\in(0,\infty)$. Then, $\Delta_\lambda\varphi_\xi^\lambda(x)=\xi^2\varphi_\xi^\lambda(x)$, see \cite{BetanDT}. These functions
will play the role of the exponentials $e^{-ix\xi}$  in the case of the
Laplacian.

We also consider the Bessel operator
$$S_\lambda=-\frac{d^2}{dx^2}+\frac{\lambda^2-\lambda}{x^2},$$
which is positive and symmetric in $L^2((0,\infty),dx)$. Observe that the potential $V(x)=\frac{\lambda^2-\lambda}{x^2}$ is a locally bounded function. If we let $\psi_\xi^\lambda(x)= x^\lambda\varphi_\xi^\lambda(x)$ then $S_\lambda\psi_\xi^\lambda(x)=\xi^2\psi_\xi^\lambda(x)$, see \cite{BetanDT}. The Hankel transform
$$f\longmapsto\int_0^\infty\psi^\lambda(\xi x)f(x)\,dx$$
is a unitary transformation in $L^2((0,\infty),dx)$, see
\cite[Chapter~8]{Titchmarsh}. On the other hand, it is known that
for any compact subset $K\subset(0,\infty)$ and $k\in\mathbb{N}_0$,
there exist a nonnegative number $\varepsilon=\varepsilon_{K,k}$ and
a constant $C=C_{K,k}$ such that $\|\psi_\xi^\lambda(x)\|_{L^\infty(K,dx)}\le C$, and $\|\frac{d^k}{dx^k}\psi_\xi^\lambda(x)\|_{L^\infty(K,dx)}\le C|\xi|^{\varepsilon}$, see \cite{Lebedev}. Therefore parallel to the case of the Laplacian
we can define  the heat semigroup as $$ e^{-tS_\lambda}f(x) =
\int_0^\infty e^{-t\xi^2}c_\xi(f)\psi^\lambda_\xi(x)\, d\xi, $$
where  $\displaystyle c_\xi(f)=\int_0^\infty f(x)\psi^\lambda_\xi(x)\,dx$. Moreover,
\begin{equation*}\label{etL}
\abs{e^{-tS_\lambda }f(x)}\leq\int_0^\infty\abs{e^{-t\xi^2}c_\xi(f)
\psi^\lambda_\xi(x)}\,d\xi\le Ct^{-1/4}\norm{f}_{L^2(K,dx)},\quad x\in K,
\end{equation*}
so the integral that defines $e^{-tS_\lambda}f(x)$ is absolutely convergent in
$K\times (0,T)$ with $T>0.$ Since $e^{-tS_\lambda}$ is positivity-preserving (see \cite{BetanDT}), we can follow step by step the arguments we gave for the case of the classical Laplacian to derive Theorem A for the operator $(S_\lambda)^\sigma$.

In order to get Theorem A for $(\Delta_\lambda)^\sigma$ we apply the transference method. Indeed, an obvious modification of Lemma \ref{Lem:change measure} is applied with $M(x) =  x^\lambda$ to get $(\Delta_\lambda)^\sigma=U^{-1}\circ(S_\lambda)^\sigma\circ U$.

\bigskip

\noindent\textbf{Acknowledgements.} We are very grateful to Jos\'e L. Torrea for many fruitful discussions.

%%%%%%%%%%%%%%%%%%%%%%%%%%%%%%%%%%%%%%%%%%%%%%%%%%%%%%

%%%%%%%%%%%%%%%%%%%%%%%%%%%%%%%%%%%%%%%%%%%%%%%%%%%%%%

%%%%%%%%%%%%%%%%%%%%%%%%%%%%%%%%%%%%%%%%%%%%%%%%%%%%%%
\end{document}